\documentclass[11pt]{article}
\usepackage[utf8]{inputenc}
\usepackage{amsmath,amsfonts,amsxtra,amssymb,latexsym, amscd,amsthm}
\usepackage[justification=centering]{caption}
\usepackage{hyperref}
\usepackage{subcaption}
\usepackage{graphicx}
\usepackage[table,xcdraw]{xcolor}
\usepackage{mathptmx}
\usepackage{setspace}%
\usepackage{enumitem}
\usepackage{tikz}
\usepackage{scalerel}

\usepackage[rmargin=0.35in,lmargin=0.35in,tmargin=0.35in,bmargin=0.5in,paperwidth=5.5in,paperheight=8.5in]{geometry}

\setlist[itemize]{leftmargin=*} 
\setlist[enumerate]{leftmargin=*}

\hypersetup{
	colorlinks=true,
	linkcolor=blue,
	citecolor=blue,
	filecolor=magenta,      
	urlcolor=cyan,
}

\newtheorem{theorem}{Theorem}[section]
\newtheorem{lemma}[theorem]{Lemma}
\newtheorem*{theorem*}{Theorem}

\theoremstyle{definition}
\newtheorem{defi}[theorem]{Definition}

\newcommand{\Z}{\mathbb{Z}}
\newcommand{\N}{\mathbb{N}}

\title{\textbf{\MakeUppercase{A square-grid coloring problem}}}
\author{\bf Matthew Kahle\\
The Ohio State University\\
e-mail address: mkahle@math.osu.edu \\
\bf Francisco Martinez-Figueroa \\
The Ohio State University\\
e-mail address: martinezfigueroa.2@osu.edu \\
\bf Alexander Soifer\\
University of Colorado at Colorado Springs\\
e-mail address: asoifer@uccs.edu
}

\date{\today}

\newcommand{\ExamplesA}{
\begin{figure}
\Large
		\centering
		\scalebox{0.65}{
		\begin{tikzpicture}
		    \draw (0,0)--(2,0);
		    \draw (0,1)--(2,1);
		    \draw (0,2)--(2,2);
		    \draw (0,0)--(0,2);
		    \draw (1,0)--(1,2);
		    \draw (2,0)--(2,2);
		    \node at (0.5, 0.5) {3};
		    \node at (1.5, 0.5) {3};
		    \node at (0.5, 1.5) {1};
		    \node at (1.5, 1.5) {2};
		\end{tikzpicture}
		\hspace{0.2in}
		\begin{tikzpicture}
		    \draw (0,0)--(3,0);
		    \draw (0,1)--(3,1);
		    \draw (0,2)--(3,2);
		    \draw (0,3)--(3,3);
		    \draw (0,0)--(0,3);
		    \draw (1,0)--(1,3);
		    \draw (2,0)--(2,3);
		    \draw (3,0)--(3,3);
		    \node at (0.5, 0.5) {2};
		    \node at (1.5, 0.5) {3};
		    \node at (2.5, 0.5) {4};
		    \node at (0.5, 1.5) {4};
		    \node at (1.5, 1.5) {5};
		    \node at (2.5, 1.5) {1};
		    \node at (0.5, 2.5) {1};
		    \node at (1.5, 2.5) {2};
		    \node at (2.5, 2.5) {3};
		\end{tikzpicture}
		\hspace{0.2in}
		\begin{tikzpicture}
		    \draw(0,0)--(4,0);
		    \draw (0,1)--(4,1);
		    \draw (0,2)--(4,2);
		    \draw (0,3)--(4,3);
		    \draw (0,4)--(4,4);
		    \draw (0,0)--(0,4);
		    \draw (1,0)--(1,4);
		    \draw (2,0)--(2,4);
		    \draw (3,0)--(3,4);
		    \draw (4,0)--(4,4);
		    \node at (0.5, 0.5) {6};
		    \node at (1.5, 0.5) {4};
		    \node at (2.5, 0.5) {2};
		    \node at (3.5, 0.5) {7};
		    \node at (0.5, 1.5) {3};
		    \node at (1.5, 1.5) {1};
		    \node at (2.5, 1.5) {5};
		    \node at (3.5, 1.5) {4};
		    \node at (0.5, 2.5) {5};
		    \node at (1.5, 2.5) {6};
		    \node at (2.5, 2.5) {7};
		    \node at (3.5, 2.5) {1};
		    \node at (0.5, 3.5) {1};
		    \node at (1.5, 3.5) {2};
		    \node at (2.5, 3.5) {3};
		    \node at (3.5, 3.5) {4};
		\end{tikzpicture}	
		}
		\caption{Optimal colorings for $n=2,3,4$.}
		\label{fig:examples1}
\end{figure}
}

\newcommand{\ExamplesB}{
\begin{figure}
\Large
\centering
\begin{subfigure}[a]{\textwidth}
\centering
\scalebox{0.6}{
        \begin{tikzpicture}
        \draw (0,0)--(5,0);
        \draw (0,0)--(0,5);
        \draw (0,1)--(5,1);
        \draw (1,0)--(1,5);
        \draw (0,2)--(5,2);
        \draw (2,0)--(2,5);
        \draw (0,3)--(5,3);
        \draw (3,0)--(3,5);
        \draw (0,4)--(5,4);
        \draw (4,0)--(4,5);
        \draw (0,5)--(5,5);
        \draw (5,0)--(5,5);
        \node at (0.5,0.5){2};
        \node at (0.5,1.5){5};
        \node at (0.5,2.5){3};
        \node at (0.5,3.5){6};
        \node at (0.5,4.5){1};
        \node at (1.5,0.5){6};
        \node at (1.5,1.5){8};
        \node at (1.5,2.5){1};
        \node at (1.5,3.5){7};
        \node at (1.5,4.5){2};
        \node at (2.5,0.5){9};
        \node at (2.5,1.5){2};
        \node at (2.5,2.5){4};
        \node at (2.5,3.5){8};
        \node at (2.5,4.5){3};
        \node at (3.5,0.5){6};
        \node at (3.5,1.5){5};
        \node at (3.5,2.5){7};
        \node at (3.5,3.5){9};
        \node at (3.5,4.5){4};
        \node at (4.5,0.5){4};
        \node at (4.5,1.5){9};
        \node at (4.5,2.5){3};
        \node at (4.5,3.5){1};
        \node at (4.5,4.5){5};
        \end{tikzpicture}}\hspace{0.2in}
\scalebox{0.6}{
\begin{tikzpicture}
        \draw (0,0)--(6,0);
        \draw (0,0)--(0,6);
        \draw (0,1)--(6,1);
        \draw (1,0)--(1,6);
        \draw (0,2)--(6,2);
        \draw (2,0)--(2,6);
        \draw (0,3)--(6,3);
        \draw (3,0)--(3,6);
        \draw (0,4)--(6,4);
        \draw (4,0)--(4,6);
        \draw (0,5)--(6,5);
        \draw (5,0)--(5,6);
        \draw (0,6)--(6,6);
        \draw (6,0)--(6,6);
        \node at (0.5,0.5){1};
        \node at (0.5,1.5){4};
        \node at (0.5,2.5){9};
        \node at (0.5,3.5){5};
        \node at (0.5,4.5){7};
        \node at (0.5,5.5){1};
        \node at (1.5,0.5){8};
        \node at (1.5,1.5){6};
        \node at (1.5,2.5){7};
        \node at (1.5,3.5){10};
        \node at (1.5,4.5){8};
        \node at (1.5,5.5){2};
        \node at (2.5,0.5){5};
        \node at (2.5,1.5){2};
        \node at (2.5,2.5){11};
        \node at (2.5,3.5){6};
        \node at (2.5,4.5){9};
        \node at (2.5,5.5){3};
        \node at (3.5,0.5){1};
        \node at (3.5,1.5){9};
        \node at (3.5,2.5){8};
        \node at (3.5,3.5){3};
        \node at (3.5,4.5){10};
        \node at (3.5,5.5){4};
        \node at (4.5,0.5){3};
        \node at (4.5,1.5){11};
        \node at (4.5,2.5){4};
        \node at (4.5,3.5){7};
        \node at (4.5,4.5){11};
        \node at (4.5,5.5){5};
        \node at (5.5,0.5){5};
        \node at (5.5,1.5){7};
        \node at (5.5,2.5){2};
        \node at (5.5,3.5){10};
        \node at (5.5,4.5){1};
        \node at (5.5,5.5){6};
\end{tikzpicture}}
\end{subfigure}\vspace{0.2in}
\begin{subfigure}[b]{\textwidth}
\centering
\scalebox{0.55}{
\begin{tikzpicture}
\centering
        \draw (0,0)--(7,0);
        \draw (0,0)--(0,7);
        \draw (0,1)--(7,1);
        \draw (1,0)--(1,7);
        \draw (0,2)--(7,2);
        \draw (2,0)--(2,7);
        \draw (0,3)--(7,3);
        \draw (3,0)--(3,7);
        \draw (0,4)--(7,4);
        \draw (4,0)--(4,7);
        \draw (0,5)--(7,5);
        \draw (5,0)--(5,7);
        \draw (0,6)--(7,6);
        \draw (6,0)--(6,7);
        \draw (0,7)--(7,7);
        \draw (7,0)--(7,7);
        \node at (0.5,0.5){3};
        \node at (0.5,1.5){7};
        \node at (0.5,2.5){2};
        \node at (0.5,3.5){6};
        \node at (0.5,4.5){12};
        \node at (0.5,5.5){8};
        \node at (0.5,6.5){1};
        \node at (1.5,0.5){1};
        \node at (1.5,1.5){10};
        \node at (1.5,2.5){13};
        \node at (1.5,3.5){11};
        \node at (1.5,4.5){1};
        \node at (1.5,5.5){9};
        \node at (1.5,6.5){2};
        \node at (2.5,0.5){5};
        \node at (2.5,1.5){2};
        \node at (2.5,2.5){4};
        \node at (2.5,3.5){9};
        \node at (2.5,4.5){6};
        \node at (2.5,5.5){10};
        \node at (2.5,6.5){3};
        \node at (3.5,0.5){7};
        \node at (3.5,1.5){11};
        \node at (3.5,2.5){8};
        \node at (3.5,3.5){5};
        \node at (3.5,4.5){3};
        \node at (3.5,5.5){11};
        \node at (3.5,6.5){4};
        \node at (4.5,0.5){12};
        \node at (4.5,1.5){5};
        \node at (4.5,2.5){3};
        \node at (4.5,3.5){13};
        \node at (4.5,4.5){9};
        \node at (4.5,5.5){12};
        \node at (4.5,6.5){5};
        \node at (5.5,0.5){2};
        \node at (5.5,1.5){10};
        \node at (5.5,2.5){12};
        \node at (5.5,3.5){8};
        \node at (5.5,4.5){7};
        \node at (5.5,5.5){13};
        \node at (5.5,6.5){6};
        \node at (6.5,0.5){8};
        \node at (6.5,1.5){6};
        \node at (6.5,2.5){4};
        \node at (6.5,3.5){10};
        \node at (6.5,4.5){4};
        \node at (6.5,5.5){1};
        \node at (6.5,6.5){7};
\end{tikzpicture}}\hspace{0.2in}
\scalebox{0.55}{
\begin{tikzpicture}
        \draw (0,0)--(8,0);
        \draw (0,0)--(0,8);
        \draw (0,1)--(8,1);
        \draw (1,0)--(1,8);
        \draw (0,2)--(8,2);
        \draw (2,0)--(2,8);
        \draw (0,3)--(8,3);
        \draw (3,0)--(3,8);
        \draw (0,4)--(8,4);
        \draw (4,0)--(4,8);
        \draw (0,5)--(8,5);
        \draw (5,0)--(5,8);
        \draw (0,6)--(8,6);
        \draw (6,0)--(6,8);
        \draw (0,7)--(8,7);
        \draw (7,0)--(7,8);
        \draw (0,8)--(8,8);
        \draw (8,0)--(8,8);
        \node at (0.5,0.5){1};
        \node at (0.5,1.5){6};
        \node at (0.5,2.5){2};
        \node at (0.5,3.5){12};
        \node at (0.5,4.5){3};
        \node at (0.5,5.5){15};
        \node at (0.5,6.5){9};
        \node at (0.5,7.5){1};
        \node at (1.5,0.5){14};
        \node at (1.5,1.5){10};
        \node at (1.5,2.5){4};
        \node at (1.5,3.5){9};
        \node at (1.5,4.5){6};
        \node at (1.5,5.5){12};
        \node at (1.5,6.5){10};
        \node at (1.5,7.5){2};
        \node at (2.5,0.5){4};
        \node at (2.5,1.5){15};
        \node at (2.5,2.5){8};
        \node at (2.5,3.5){11};
        \node at (2.5,4.5){4};
        \node at (2.5,5.5){7};
        \node at (2.5,6.5){11};
        \node at (2.5,7.5){3};
        \node at (3.5,0.5){1};
        \node at (3.5,1.5){3};
        \node at (3.5,2.5){5};
        \node at (3.5,3.5){15};
        \node at (3.5,4.5){13};
        \node at (3.5,5.5){1};
        \node at (3.5,6.5){12};
        \node at (3.5,7.5){4};
        \node at (4.5,0.5){5};
        \node at (4.5,1.5){14};
        \node at (4.5,2.5){11};
        \node at (4.5,3.5){6};
        \node at (4.5,4.5){8};
        \node at (4.5,5.5){10};
        \node at (4.5,6.5){13};
        \node at (4.5,7.5){5};
        \node at (5.5,0.5){12};
        \node at (5.5,1.5){8};
        \node at (5.5,2.5){2};
        \node at (5.5,3.5){13};
        \node at (5.5,4.5){3};
        \node at (5.5,5.5){7};
        \node at (5.5,6.5){14};
        \node at (5.5,7.5){6};
        \node at (6.5,0.5){14};
        \node at (6.5,1.5){9};
        \node at (6.5,2.5){5};
        \node at (6.5,3.5){7};
        \node at (6.5,4.5){9};
        \node at (6.5,5.5){2};
        \node at (6.5,6.5){15};
        \node at (6.5,7.5){7};
        \node at (7.5,0.5){2};
        \node at (7.5,1.5){10};
        \node at (7.5,2.5){10};
        \node at (7.5,3.5){3};
        \node at (7.5,4.5){13};
        \node at (7.5,5.5){11};
        \node at (7.5,6.5){1};
        \node at (7.5,7.5){8};
\end{tikzpicture}}
\end{subfigure}\vspace{0.2in}
\begin{subfigure}[c]{\textwidth}
\centering
\scalebox{0.55}{
\begin{tikzpicture}
\centering
       \draw (0,0)--(9,0);
\draw (0,0)--(0,9);
\draw (0,1)--(9,1);
\draw (1,0)--(1,9);
\draw (0,2)--(9,2);
\draw (2,0)--(2,9);
\draw (0,3)--(9,3);
\draw (3,0)--(3,9);
\draw (0,4)--(9,4);
\draw (4,0)--(4,9);
\draw (0,5)--(9,5);
\draw (5,0)--(5,9);
\draw (0,6)--(9,6);
\draw (6,0)--(6,9);
\draw (0,7)--(9,7);
\draw (7,0)--(7,9);
\draw (0,8)--(9,8);
\draw (8,0)--(8,9);
\draw (0,9)--(9,9);
\draw (9,0)--(9,9);
\node at (0.5,0.5){6};
\node at (0.5,1.5){3};
\node at (0.5,2.5){13};
\node at (0.5,3.5){5};
\node at (0.5,4.5){4};
\node at (0.5,5.5){12};
\node at (0.5,6.5){11};
\node at (0.5,7.5){17};
\node at (0.5,8.5){9};
\node at (1.5,0.5){13};
\node at (1.5,1.5){16};
\node at (1.5,2.5){4};
\node at (1.5,3.5){14};
\node at (1.5,4.5){8};
\node at (1.5,5.5){5};
\node at (1.5,6.5){6};
\node at (1.5,7.5){16};
\node at (1.5,8.5){2};
\node at (2.5,0.5){2};
\node at (2.5,1.5){11};
\node at (2.5,2.5){10};
\node at (2.5,3.5){7};
\node at (2.5,4.5){2};
\node at (2.5,5.5){15};
\node at (2.5,6.5){9};
\node at (2.5,7.5){1};
\node at (2.5,8.5){5};
\node at (3.5,0.5){4};
\node at (3.5,1.5){11};
\node at (3.5,2.5){15};
\node at (3.5,3.5){16};
\node at (3.5,4.5){12};
\node at (3.5,5.5){17};
\node at (3.5,6.5){8};
\node at (3.5,7.5){11};
\node at (3.5,8.5){3};
\node at (4.5,0.5){1};
\node at (4.5,1.5){2};
\node at (4.5,2.5){6};
\node at (4.5,3.5){8};
\node at (4.5,4.5){3};
\node at (4.5,5.5){10};
\node at (4.5,6.5){12};
\node at (4.5,7.5){14};
\node at (4.5,8.5){1};
\node at (5.5,0.5){5};
\node at (5.5,1.5){10};
\node at (5.5,2.5){1};
\node at (5.5,3.5){15};
\node at (5.5,4.5){14};
\node at (5.5,5.5){16};
\node at (5.5,6.5){9};
\node at (5.5,7.5){10};
\node at (5.5,8.5){8};
\node at (6.5,0.5){7};
\node at (6.5,1.5){6};
\node at (6.5,2.5){17};
\node at (6.5,3.5){3};
\node at (6.5,4.5){9};
\node at (6.5,5.5){5};
\node at (6.5,6.5){11};
\node at (6.5,7.5){13};
\node at (6.5,8.5){7};
\node at (7.5,0.5){17};
\node at (7.5,1.5){14};
\node at (7.5,2.5){2};
\node at (7.5,3.5){3};
\node at (7.5,4.5){4};
\node at (7.5,5.5){17};
\node at (7.5,6.5){7};
\node at (7.5,7.5){15};
\node at (7.5,8.5){4};
\node at (8.5,0.5){8};
\node at (8.5,1.5){13};
\node at (8.5,2.5){12};
\node at (8.5,3.5){7};
\node at (8.5,4.5){9};
\node at (8.5,5.5){13};
\node at (8.5,6.5){1};
\node at (8.5,7.5){12};
\node at (8.5,8.5){6};
\end{tikzpicture}}\hspace{0.2in}
\scalebox{0.55}{
\begin{tikzpicture}
       \draw (0,0)--(10,0);
\draw (0,0)--(0,10);
\draw (0,1)--(10,1);
\draw (1,0)--(1,10);
\draw (0,2)--(10,2);
\draw (2,0)--(2,10);
\draw (0,3)--(10,3);
\draw (3,0)--(3,10);
\draw (0,4)--(10,4);
\draw (4,0)--(4,10);
\draw (0,5)--(10,5);
\draw (5,0)--(5,10);
\draw (0,6)--(10,6);
\draw (6,0)--(6,10);
\draw (0,7)--(10,7);
\draw (7,0)--(7,10);
\draw (0,8)--(10,8);
\draw (8,0)--(8,10);
\draw (0,9)--(10,9);
\draw (9,0)--(9,10);
\draw (0,10)--(10,10);
\draw (10,0)--(10,10);
\node at (0.5,0.5){15};
\node at (0.5,1.5){1};
\node at (0.5,2.5){11};
\node at (0.5,3.5){14};
\node at (0.5,4.5){9};
\node at (0.5,5.5){3};
\node at (0.5,6.5){14};
\node at (0.5,7.5){18};
\node at (0.5,8.5){16};
\node at (0.5,9.5){1};
\node at (1.5,0.5){14};
\node at (1.5,1.5){12};
\node at (1.5,2.5){15};
\node at (1.5,3.5){8};
\node at (1.5,4.5){7};
\node at (1.5,5.5){15};
\node at (1.5,6.5){6};
\node at (1.5,7.5){19};
\node at (1.5,8.5){2};
\node at (1.5,9.5){13};
\node at (2.5,0.5){1};
\node at (2.5,1.5){8};
\node at (2.5,2.5){9};
\node at (2.5,3.5){17};
\node at (2.5,4.5){14};
\node at (2.5,5.5){4};
\node at (2.5,6.5){11};
\node at (2.5,7.5){9};
\node at (2.5,8.5){1};
\node at (2.5,9.5){6};
\node at (3.5,0.5){17};
\node at (3.5,1.5){16};
\node at (3.5,2.5){5};
\node at (3.5,3.5){3};
\node at (3.5,4.5){2};
\node at (3.5,5.5){10};
\node at (3.5,6.5){5};
\node at (3.5,7.5){18};
\node at (3.5,8.5){10};
\node at (3.5,9.5){8};
\node at (4.5,0.5){19};
\node at (4.5,1.5){11};
\node at (4.5,2.5){17};
\node at (4.5,3.5){6};
\node at (4.5,4.5){18};
\node at (4.5,5.5){17};
\node at (4.5,6.5){15};
\node at (4.5,7.5){13};
\node at (4.5,8.5){3};
\node at (4.5,9.5){11};
\node at (5.5,0.5){1};
\node at (5.5,1.5){7};
\node at (5.5,2.5){13};
\node at (5.5,3.5){10};
\node at (5.5,4.5){7};
\node at (5.5,5.5){4};
\node at (5.5,6.5){18};
\node at (5.5,7.5){11};
\node at (5.5,8.5){12};
\node at (5.5,9.5){10};
\node at (6.5,0.5){18};
\node at (6.5,1.5){3};
\node at (6.5,2.5){16};
\node at (6.5,3.5){15};
\node at (6.5,4.5){19};
\node at (6.5,5.5){5};
\node at (6.5,6.5){12};
\node at (6.5,7.5){2};
\node at (6.5,8.5){7};
\node at (6.5,9.5){16};
\node at (7.5,0.5){8};
\node at (7.5,1.5){8};
\node at (7.5,2.5){6};
\node at (7.5,3.5){2};
\node at (7.5,4.5){10};
\node at (7.5,5.5){14};
\node at (7.5,6.5){13};
\node at (7.5,7.5){9};
\node at (7.5,8.5){6};
\node at (7.5,9.5){12};
\node at (8.5,0.5){4};
\node at (8.5,1.5){13};
\node at (8.5,2.5){5};
\node at (8.5,3.5){8};
\node at (8.5,4.5){9};
\node at (8.5,5.5){16};
\node at (8.5,6.5){19};
\node at (8.5,7.5){4};
\node at (8.5,8.5){2};
\node at (8.5,9.5){17};
\node at (9.5,0.5){6};
\node at (9.5,1.5){15};
\node at (9.5,2.5){14};
\node at (9.5,3.5){19};
\node at (9.5,4.5){12};
\node at (9.5,5.5){4};
\node at (9.5,6.5){3};
\node at (9.5,7.5){1};
\node at (9.5,8.5){5};
\node at (9.5,9.5){7};
\end{tikzpicture}}
\end{subfigure}
\caption{Optimal colorings for $5\leq n\leq 10$, found with computer-aided search.}
\label{fig:examples2}
\end{figure}
 }

\newcommand{\ExamplesC}{
\begin{figure}
\Large
\centering
\scalebox{0.64}{
\begin{tikzpicture}
        \draw (0,0)--(19,0);
    \draw (0,0)--(0,19);
    \draw (0,1)--(19,1);
    \draw (1,0)--(1,19);
    \draw (0,2)--(19,2);
    \draw (2,0)--(2,19);
    \draw (0,3)--(19,3);
    \draw (3,0)--(3,19);
    \draw (0,4)--(19,4);
    \draw (4,0)--(4,19);
    \draw (0,5)--(19,5);
    \draw (5,0)--(5,19);
    \draw (0,6)--(19,6);
    \draw (6,0)--(6,19);
    \draw (0,7)--(19,7);
    \draw (7,0)--(7,19);
    \draw (0,8)--(19,8);
    \draw (8,0)--(8,19);
    \draw (0,9)--(19,9);
    \draw (9,0)--(9,19);
    \draw (0,10)--(19,10);
    \draw (10,0)--(10,19);
    \draw (0,11)--(19,11);
    \draw (11,0)--(11,19);
    \draw (0,12)--(19,12);
    \draw (12,0)--(12,19);
    \draw (0,13)--(19,13);
    \draw (13,0)--(13,19);
    \draw (0,14)--(19,14);
    \draw (14,0)--(14,19);
    \draw (0,15)--(19,15);
    \draw (15,0)--(15,19);
    \draw (0,16)--(19,16);
    \draw (16,0)--(16,19);
    \draw (0,17)--(19,17);
    \draw (17,0)--(17,19);
    \draw (0,18)--(19,18);
    \draw (18,0)--(18,19);
    \draw (0,19)--(19,19);
    \draw (19,0)--(19,19);
    \node at (0.5,0.5){31};
    \node at (0.5,1.5){23};
    \node at (0.5,2.5){19};
    \node at (0.5,3.5){9};
    \node at (0.5,4.5){14};
    \node at (0.5,5.5){18};
    \node at (0.5,6.5){16};
    \node at (0.5,7.5){19};
    \node at (0.5,8.5){1};
    \node at (0.5,9.5){22};
    \node at (0.5,10.5){12};
    \node at (0.5,11.5){16};
    \node at (0.5,12.5){9};
    \node at (0.5,13.5){1};
    \node at (0.5,14.5){10};
    \node at (0.5,15.5){25};
    \node at (0.5,16.5){2};
    \node at (0.5,17.5){20};
    \node at (0.5,18.5){1};
    \node at (1.5,0.5){7};
    \node at (1.5,1.5){1};
    \node at (1.5,2.5){30};
    \node at (1.5,3.5){20};
    \node at (1.5,4.5){3};
    \node at (1.5,5.5){1};
    \node at (1.5,6.5){34};
    \node at (1.5,7.5){26};
    \node at (1.5,8.5){12};
    \node at (1.5,9.5){32};
    \node at (1.5,10.5){9};
    \node at (1.5,11.5){24};
    \node at (1.5,12.5){29};
    \node at (1.5,13.5){5};
    \node at (1.5,14.5){16};
    \node at (1.5,15.5){4};
    \node at (1.5,16.5){19};
    \node at (1.5,17.5){21};
    \node at (1.5,18.5){2};
    \node at (2.5,0.5){9};
    \node at (2.5,1.5){4};
    \node at (2.5,2.5){13};
    \node at (2.5,3.5){24};
    \node at (2.5,4.5){11};
    \node at (2.5,5.5){16};
    \node at (2.5,6.5){32};
    \node at (2.5,7.5){28};
    \node at (2.5,8.5){15};
    \node at (2.5,9.5){2};
    \node at (2.5,10.5){30};
    \node at (2.5,11.5){6};
    \node at (2.5,12.5){34};
    \node at (2.5,13.5){7};
    \node at (2.5,14.5){14};
    \node at (2.5,15.5){6};
    \node at (2.5,16.5){13};
    \node at (2.5,17.5){22};
    \node at (2.5,18.5){3};
    \node at (3.5,0.5){2};
    \node at (3.5,1.5){26};
    \node at (3.5,2.5){21};
    \node at (3.5,3.5){12};
    \node at (3.5,4.5){35};
    \node at (3.5,5.5){20};
    \node at (3.5,6.5){6};
    \node at (3.5,7.5){35};
    \node at (3.5,8.5){5};
    \node at (3.5,9.5){12};
    \node at (3.5,10.5){14};
    \node at (3.5,11.5){22};
    \node at (3.5,12.5){9};
    \node at (3.5,13.5){13};
    \node at (3.5,14.5){27};
    \node at (3.5,15.5){9};
    \node at (3.5,16.5){18};
    \node at (3.5,17.5){23};
    \node at (3.5,18.5){4};
    \node at (4.5,0.5){16};
    \node at (4.5,1.5){9};
    \node at (4.5,2.5){3};
    \node at (4.5,3.5){6};
    \node at (4.5,4.5){2};
    \node at (4.5,5.5){23};
    \node at (4.5,6.5){33};
    \node at (4.5,7.5){19};
    \node at (4.5,8.5){27};
    \node at (4.5,9.5){4};
    \node at (4.5,10.5){28};
    \node at (4.5,11.5){20};
    \node at (4.5,12.5){17};
    \node at (4.5,13.5){8};
    \node at (4.5,14.5){30};
    \node at (4.5,15.5){25};
    \node at (4.5,16.5){15};
    \node at (4.5,17.5){24};
    \node at (4.5,18.5){5};
    \node at (5.5,0.5){33};
    \node at (5.5,1.5){35};
    \node at (5.5,2.5){27};
    \node at (5.5,3.5){23};
    \node at (5.5,4.5){13};
    \node at (5.5,5.5){29};
    \node at (5.5,6.5){4};
    \node at (5.5,7.5){10};
    \node at (5.5,8.5){15};
    \node at (5.5,9.5){22};
    \node at (5.5,10.5){2};
    \node at (5.5,11.5){33};
    \node at (5.5,12.5){12};
    \node at (5.5,13.5){18};
    \node at (5.5,14.5){34};
    \node at (5.5,15.5){12};
    \node at (5.5,16.5){23};
    \node at (5.5,17.5){25};
    \node at (5.5,18.5){6};
    \node at (6.5,0.5){8};
    \node at (6.5,1.5){21};
    \node at (6.5,2.5){6};
    \node at (6.5,3.5){16};
    \node at (6.5,4.5){20};
    \node at (6.5,5.5){15};
    \node at (6.5,6.5){32};
    \node at (6.5,7.5){20};
    \node at (6.5,8.5){26};
    \node at (6.5,9.5){8};
    \node at (6.5,10.5){11};
    \node at (6.5,11.5){18};
    \node at (6.5,12.5){27};
    \node at (6.5,13.5){31};
    \node at (6.5,14.5){10};
    \node at (6.5,15.5){30};
    \node at (6.5,16.5){3};
    \node at (6.5,17.5){26};
    \node at (6.5,18.5){7};
    \node at (7.5,0.5){3};
    \node at (7.5,1.5){31};
    \node at (7.5,2.5){11};
    \node at (7.5,3.5){2};
    \node at (7.5,4.5){27};
    \node at (7.5,5.5){7};
    \node at (7.5,6.5){3};
    \node at (7.5,7.5){34};
    \node at (7.5,8.5){24};
    \node at (7.5,9.5){28};
    \node at (7.5,10.5){22};
    \node at (7.5,11.5){5};
    \node at (7.5,12.5){33};
    \node at (7.5,13.5){24};
    \node at (7.5,14.5){21};
    \node at (7.5,15.5){18};
    \node at (7.5,16.5){24};
    \node at (7.5,17.5){27};
    \node at (7.5,18.5){8};
    \node at (8.5,0.5){10};
    \node at (8.5,1.5){25};
    \node at (8.5,2.5){27};
    \node at (8.5,3.5){34};
    \node at (8.5,4.5){21};
    \node at (8.5,5.5){4};
    \node at (8.5,6.5){35};
    \node at (8.5,7.5){17};
    \node at (8.5,8.5){10};
    \node at (8.5,9.5){7};
    \node at (8.5,10.5){16};
    \node at (8.5,11.5){25};
    \node at (8.5,12.5){22};
    \node at (8.5,13.5){19};
    \node at (8.5,14.5){31};
    \node at (8.5,15.5){6};
    \node at (8.5,16.5){1};
    \node at (8.5,17.5){28};
    \node at (8.5,18.5){9};
    \node at (9.5,0.5){12};
    \node at (9.5,1.5){21};
    \node at (9.5,2.5){16};
    \node at (9.5,3.5){8};
    \node at (9.5,4.5){1};
    \node at (9.5,5.5){11};
    \node at (9.5,6.5){25};
    \node at (9.5,7.5){14};
    \node at (9.5,8.5){23};
    \node at (9.5,9.5){32};
    \node at (9.5,10.5){29};
    \node at (9.5,11.5){7};
    \node at (9.5,12.5){18};
    \node at (9.5,13.5){25};
    \node at (9.5,14.5){28};
    \node at (9.5,15.5){10};
    \node at (9.5,16.5){27};
    \node at (9.5,17.5){29};
    \node at (9.5,18.5){10};
    \node at (10.5,0.5){1};
    \node at (10.5,1.5){29};
    \node at (10.5,2.5){3};
    \node at (10.5,3.5){32};
    \node at (10.5,4.5){25};
    \node at (10.5,5.5){33};
    \node at (10.5,6.5){29};
    \node at (10.5,7.5){31};
    \node at (10.5,8.5){34};
    \node at (10.5,9.5){14};
    \node at (10.5,10.5){19};
    \node at (10.5,11.5){11};
    \node at (10.5,12.5){29};
    \node at (10.5,13.5){17};
    \node at (10.5,14.5){21};
    \node at (10.5,15.5){5};
    \node at (10.5,16.5){32};
    \node at (10.5,17.5){30};
    \node at (10.5,18.5){11};
    \node at (11.5,0.5){33};
    \node at (11.5,1.5){8};
    \node at (11.5,2.5){5};
    \node at (11.5,3.5){13};
    \node at (11.5,4.5){34};
    \node at (11.5,5.5){30};
    \node at (11.5,6.5){20};
    \node at (11.5,7.5){2};
    \node at (11.5,8.5){4};
    \node at (11.5,9.5){8};
    \node at (11.5,10.5){20};
    \node at (11.5,11.5){5};
    \node at (11.5,12.5){26};
    \node at (11.5,13.5){22};
    \node at (11.5,14.5){7};
    \node at (11.5,15.5){23};
    \node at (11.5,16.5){17};
    \node at (11.5,17.5){31};
    \node at (11.5,18.5){12};
    \node at (12.5,0.5){5};
    \node at (12.5,1.5){19};
    \node at (12.5,2.5){28};
    \node at (12.5,3.5){35};
    \node at (12.5,4.5){11};
    \node at (12.5,5.5){16};
    \node at (12.5,6.5){4};
    \node at (12.5,7.5){20};
    \node at (12.5,8.5){18};
    \node at (12.5,9.5){23};
    \node at (12.5,10.5){11};
    \node at (12.5,11.5){9};
    \node at (12.5,12.5){23};
    \node at (12.5,13.5){35};
    \node at (12.5,14.5){30};
    \node at (12.5,15.5){28};
    \node at (12.5,16.5){11};
    \node at (12.5,17.5){32};
    \node at (12.5,18.5){13};
    \node at (13.5,0.5){14};
    \node at (13.5,1.5){32};
    \node at (13.5,2.5){13};
    \node at (13.5,3.5){26};
    \node at (13.5,4.5){3};
    \node at (13.5,5.5){28};
    \node at (13.5,6.5){27};
    \node at (13.5,7.5){17};
    \node at (13.5,8.5){24};
    \node at (13.5,9.5){30};
    \node at (13.5,10.5){26};
    \node at (13.5,11.5){33};
    \node at (13.5,12.5){21};
    \node at (13.5,13.5){14};
    \node at (13.5,14.5){4};
    \node at (13.5,15.5){17};
    \node at (13.5,16.5){13};
    \node at (13.5,17.5){33};
    \node at (13.5,18.5){14};
    \node at (14.5,0.5){11};
    \node at (14.5,1.5){25};
    \node at (14.5,2.5){27};
    \node at (14.5,3.5){22};
    \node at (14.5,4.5){29};
    \node at (14.5,5.5){6};
    \node at (14.5,6.5){6};
    \node at (14.5,7.5){26};
    \node at (14.5,8.5){32};
    \node at (14.5,9.5){21};
    \node at (14.5,10.5){16};
    \node at (14.5,11.5){31};
    \node at (14.5,12.5){4};
    \node at (14.5,13.5){24};
    \node at (14.5,14.5){7};
    \node at (14.5,15.5){33};
    \node at (14.5,16.5){28};
    \node at (14.5,17.5){34};
    \node at (14.5,18.5){15};
    \node at (15.5,0.5){31};
    \node at (15.5,1.5){8};
    \node at (15.5,2.5){17};
    \node at (15.5,3.5){30};
    \node at (15.5,4.5){33};
    \node at (15.5,5.5){3};
    \node at (15.5,6.5){15};
    \node at (15.5,7.5){31};
    \node at (15.5,8.5){35};
    \node at (15.5,9.5){15};
    \node at (15.5,10.5){13};
    \node at (15.5,11.5){1};
    \node at (15.5,12.5){15};
    \node at (15.5,13.5){8};
    \node at (15.5,14.5){35};
    \node at (15.5,15.5){10};
    \node at (15.5,16.5){18};
    \node at (15.5,17.5){35};
    \node at (15.5,18.5){16};
    \node at (16.5,0.5){1};
    \node at (16.5,1.5){12};
    \node at (16.5,2.5){21};
    \node at (16.5,3.5){11};
    \node at (16.5,4.5){15};
    \node at (16.5,5.5){18};
    \node at (16.5,6.5){7};
    \node at (16.5,7.5){20};
    \node at (16.5,8.5){24};
    \node at (16.5,9.5){9};
    \node at (16.5,10.5){31};
    \node at (16.5,11.5){5};
    \node at (16.5,12.5){17};
    \node at (16.5,13.5){2};
    \node at (16.5,14.5){29};
    \node at (16.5,15.5){14};
    \node at (16.5,16.5){26};
    \node at (16.5,17.5){1};
    \node at (16.5,18.5){17};
    \node at (17.5,0.5){4};
    \node at (17.5,1.5){14};
    \node at (17.5,2.5){9};
    \node at (17.5,3.5){10};
    \node at (17.5,4.5){10};
    \node at (17.5,5.5){32};
    \node at (17.5,6.5){1};
    \node at (17.5,7.5){14};
    \node at (17.5,8.5){2};
    \node at (17.5,9.5){24};
    \node at (17.5,10.5){22};
    \node at (17.5,11.5){34};
    \node at (17.5,12.5){19};
    \node at (17.5,13.5){7};
    \node at (17.5,14.5){12};
    \node at (17.5,15.5){20};
    \node at (17.5,16.5){10};
    \node at (17.5,17.5){2};
    \node at (17.5,18.5){18};
    \node at (18.5,0.5){23};
    \node at (18.5,1.5){15};
    \node at (18.5,2.5){7};
    \node at (18.5,3.5){28};
    \node at (18.5,4.5){12};
    \node at (18.5,5.5){19};
    \node at (18.5,6.5){15};
    \node at (18.5,7.5){30};
    \node at (18.5,8.5){5};
    \node at (18.5,9.5){23};
    \node at (18.5,10.5){10};
    \node at (18.5,11.5){8};
    \node at (18.5,12.5){6};
    \node at (18.5,13.5){17};
    \node at (18.5,14.5){3};
    \node at (18.5,15.5){25};
    \node at (18.5,16.5){13};
    \node at (18.5,17.5){3};
    \node at (18.5,18.5){19};
\end{tikzpicture}}
\caption{A $35$-complete coloring for $n=19$, found with a computer-aided search.}
\label{fig:examples3}
\end{figure}}

\begin{document}

\maketitle 

\section{Introduction}
The following problem was created by Alexander Soifer with the Colorado Mathematical Olympiad in mind, but it turned out to be too difficult for inclusion in the competition. In fact, even after several months of thinking about it, we still do not know the whole story.

Suppose $n \ge 2$, and we wish to plant $k$ different types of trees in the squares of an $n \times n$ square grid. We can have as many of each type as we want. The only rule is that every pair of types must occur adjacently somewhere in the grid. The question is: given $n$, what is the largest that $k$ can be? Denote this number by $\Gamma(n)$, and call this the \emph{complete coloring number} of the $n \times n$ grid.

A little thought shows us that $\Gamma(n) \le 2n-1$. We will prove this in the next section. The main question we are interested in is whether $\Gamma(n) = 2n-1$ for every $n \ge 2$. 

Edwards \cite{Edwards2009} showed that $\Gamma(n) = 2n-1$ for all sufficiently large $n$, but it is not clear how large $n$ must be, and the proof is not constructive.

The main contributions here are to describe optimal complete colorings for $n \le 8$, and almost optimal complete colorings for all $n$. In particular, we prove the following theorems.

\begin{theorem} \label{thm:smalln}
Suppose that $2 \le n\ \le 8$. Then $\Gamma(G_n) = 2n-1$. 
\end{theorem}

Once we have the upper bound $\Gamma(G_n)$ in hand, to prove Theorem \ref{thm:smalln} only requires producing a single example for every $n = 2, 3, \dots, 8$. These are illustrated in Figures \ref{fig:examples1} and \ref{fig:examples2}. 

\begin{theorem}\label{thm:improved_roichman}
Let $n\geq 2$, then $\Gamma(G_n)\geq 2n-9$. Moreover:
$$\Gamma(G_n) \geq \begin{cases} 2n-6, & \text{for } n\equiv 0 \text{ or }1 \pmod 4\\
2n-7, & \text{for } n\equiv 2 \pmod 4\\
2n-9, & \text{ for } n\equiv 3 \pmod 4.\\
\end{cases}$$
\end{theorem}

Our exposition is organized as follows. In Section 2, we introduce some graph-theoretic language and briefly review what we found in the literature. In Section 3 we include the best complete colorings we found for small values of $n$. In Section 4, after introducing some useful tools, we bound the achromatic number and complete coloring number of some related graphs that we use as building blocks later. In Section 5, we prove Theorem \ref{thm:improved_roichman}, using the graphs of Section 4. Finally in Section 6, we include some almost optimal complete colorings for the analogous problem for $3$-dimensional cubes.

\section{Graph-theoretic formulation}
We can restate our problem using graph--theoretic language. In what follows, let $G$ be a finite simple graph.

\begin{defi}\label{def:coloring}\leavevmode
	\begin{enumerate}
		\item We say an assignment $c:V(G)\to [k]:=\{1,2,\dots,k\}$ is a $\mathbf{k}$-\textbf{complete coloring} if for every pair of different colors $i,j\in[k]$, $i\neq j$, there exists an edge $\{u,v\}\in E(G)$ such that $c(u)=i$ and $c(v)=j$. If moreover, for every edge $\{u,v\}$ we have that $c(u)\neq c(v)$, we say $c$ is a $\mathbf{k}$-\textbf{complete proper coloring}.
		\item The \textbf{complete coloring number} of $G$, denoted $\Gamma(G)$, is the maximal integer $k$ such that there exists a $k$-complete coloring of $G$.
		\item The \textbf{achromatic number} of $G$, denoted $\Psi(G)$ is the maximal integer $k$ such that there exists a $k$-complete proper coloring of $G$.  
	\end{enumerate}	 
\end{defi}

We can also consider a rectangular grid as a graph.

\begin{defi}
	Let $m,n\in \N$, the \textbf{rectangular grid of dimensions $(m\times n)$} is the graph $R_{m\times n}$ with vertices  $V=\{(i,j)\subset \Z^2: 1\leq i\leq m, 1\leq j\leq n\}$ and edges $(i_1,j_1)\sim (i_2,j_2)$ if and only if $|i_1-i_2|+|j_1-j_2|=1$.
\end{defi}

Thus a grid with $m$ rows and $n$ columns corresponds to the graph $R_{m\times n}$, and the square grid $G_n$ can be seen as the graph $R_{n\times n}$. Note there is also a correspondence between an assignment of colors to the unit cells of a grid and a coloring of the respective graph (Figure \ref{fig:correspondence_grid}). We won't distinguish between these two representations, and in fact it will be useful to think of such a coloring as numbers inside physical cells, that can be moved and re-arranged keeping their assigned colors inside.

The complete chromatic number is sometimes referred to in the literature as the \textit{achromatic} number of a graph \cite{Harary_Hedetniemi1970}, which has been extensively studied \cite{Pavol_Miller1976,Sampathkumar1976,Farber1986,Edwards2000}. See \cite{HM97} for a 1997 survey. Computing the complete chromatic number is known to be an NP-Hard problem \cite{Yannakakis_Gavril1980}, even for trees \cite{Cairnie_Edwards1997}. The problem of computing the complete chromatic number of an $m \times n$ grid is suggested as an open problem in \cite{HM97}. 

A corollary of the more general results of Edwards \cite{Edwards2009} is that for sufficiently large $n$, it is always possible to find a $(2n-1)$-complete proper coloring for the grid $G_n$. However, it is not clear from his proof how large $n$ needs to be, and the question remains open for small values of $n$.

\begin{figure}
	\centering
	\begin{subfigure}[a]{0.45\textwidth}
		\flushright
		\includegraphics[width=0.3\textwidth]{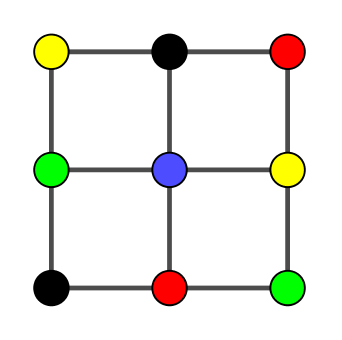}
	\end{subfigure}{\Large $\Leftrightarrow$}
    \begin{subfigure}[a]{0.45\textwidth}
    	\flushleft
    	\Large
    	\scalebox{0.6}{
		\begin{tikzpicture}
		    \draw (0,0)--(3,0);
		    \draw (0,1)--(3,1);
		    \draw (0,2)--(3,2);
		    \draw (0,3)--(3,3);
		    \draw (0,0)--(0,3);
		    \draw (1,0)--(1,3);
		    \draw (2,0)--(2,3);
		    \draw (3,0)--(3,3);
		    \node at (0.5, 0.5) {2};
		    \node at (1.5, 0.5) {3};
		    \node at (2.5, 0.5) {4};
		    \node at (0.5, 1.5) {4};
		    \node at (1.5, 1.5) {5};
		    \node at (2.5, 1.5) {1};
		    \node at (0.5, 2.5) {1};
		    \node at (1.5, 2.5) {2};
		    \node at (2.5, 2.5) {3};
		\end{tikzpicture}}
    \end{subfigure}
\caption{$G_3$ as a graph and rectangular grid}
\label{fig:correspondence_grid}
\end{figure}
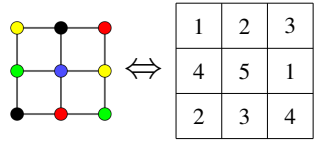

From Definition \ref{def:coloring}, it is clear that any $k$-complete proper coloring is also a $k$-complete coloring, since the later relaxes the conditions on feasible colorings. Suppose there exists a $k$-complete coloring for a graph $G$ with $m$ edges and $n$ vertices, the following upper bounds follow. Since each pair of colors in $[k]$ must be assigned to a different edge, we get $\binom{k}{2}\leq m$. Let $\Delta=\Delta(G)$ denote the maximum degree of $G$. Since each one of the $k$ colors must be a neighbor with the remaining $(k-1)$ colors, it must be assigned to at least $\lceil (k-1)/\Delta \rceil$ vertices, so $k\left\lceil\frac{k-1}{\Delta}\right\rceil\leq n$. The following lemma summarizes these results.


\begin{lemma}\label{lemma:ineq1}
The following inequalities hold for every graph $G$.
\begin{enumerate}[label={\normalfont{\arabic*}.}]
    \item  $\vspace{-1em} \Psi(G) \leq \Gamma(G )$\vspace{1em}
    \item $\displaystyle\Gamma(G )\leq \max \left\{ k: {k \choose 2} \leq |E(G)| \right\} $
    \item $\displaystyle\Gamma(G ) \leq  \max \left\{ k: \textstyle{k\left\lceil\dfrac{k-1}{\Delta(G)}\right\rceil}\leq |V(G)| \right\} $
\end{enumerate}
\end{lemma}

\bigskip

The square grid graph $G_n$ has $m=2n(n-1)=2n^2-2n$ edges, $n^2$ vertices and maximum degree $\Delta=4$. Note $$\binom{2n}{2}=n(2n-1)=2n^2-n>m\text{, while}$$ $$\binom{2n-1}{2}=(2n-1)(n-1)=2n^2-3n+1<m \text{ and}$$ $$(2n-1)\left\lceil\frac{2n-2}{4}\right\rceil\leq (2n-1)\frac{2n}{4}<n^2.$$ Hence Lemma \ref{lemma:ineq1} readily gives $\Gamma(G_n)\leq 2n-1$. 

We are interested thus in whether this upper bound can always be attained for $n\geq 2$. 

\section{Small Examples}

We found $(2n-1)$-complete colorings for $n=2, 3,$ and $4$ by hand (Figure \ref{fig:examples1}). We found $(2n-1)$-complete colorings for $5\leq n\leq 10$ (Figure \ref{fig:examples2}), with the aid of a computer. So far, we were not able to find explicit $(2n-1)$-complete colorings for any $n\geq 11$. We note that the complete colorings in Figures \ref{fig:examples1} and \ref{fig:examples2} are proper for $3 \leq n \leq 7$, and hence $$\Psi(G_n)=\Gamma(G_n)=2n-1\text{ for }3\leq n\leq 7.$$ 

For bigger values of $n$ we were only able to find almost optimal complete colorings, as we state in the following lemma.

\begin{lemma}\label{lemma:n_9_19}
    $$2n-3\leq\Gamma(G_n)\leq 2n-1\text{ \normalfont{for} }9\leq n\leq 19$$
\end{lemma}

\begin{proof}
We only need to find examples of complete colorings with $2n-3$ colorings for every $n=11, 12, \dots, 19$. These were found with the aid of a computer search. In the interest of saving space, we do not show all of these, but Figure \ref{fig:examples3} shows the case $n=19$.
\end{proof}

\ExamplesA
\ExamplesB
\ExamplesC

\section{Some Related Graphs}
\subsection{Tools}
We first introduce some definitions and lemmas that will make our results later easier to state and prove. 

\begin{defi}\leavevmode
	\begin{enumerate}
		\item Given a coloring $c:V(G)\to [k]$, and $i,j\in [k]$, we say that the coloring $c$ \textbf{realizes} the pair $\{i,j\}$, if there exist an edge $\{ v, w\} \in E(G)$ such that $c(v)=i$ and $c(w)=j$. 
		\item A \textbf{partial coloring} $c$ of $G$ is a coloring of some subgraph of $G$. That is, an assignment $c:V(H)\to [k]$ for some $k\in\N$ and some $H\subsetneq G$. If $c$ realizes all $[k]^2$ pairs of different colors, we say it is a $\mathbf{k}$\textbf{-complete partial coloring}.
		\item We say $c$ is a \textbf{$\mathbf{k}$-complete partial coloring with remainder $\mathbf{A}$}, if $c$ is a partial coloring of $G$ that realizes all pairs of different colors in $[k]^2\setminus{A}$.
		\item Given $c:V(H)\to[k]$ a $k$-complete partial coloring with remainder $A$ of $G$, we define a \textbf{$\mathbf{k}$-complete extension} of $c$, as a coloring $c':V(G)\setminus V(H)\to[k]$ that realizes all pairs of $A$. Note that then $c$ and $c'$ can be combined into a $k$-complete coloring of $G$:
		$$\bar{c}(v)=\begin{cases}
		c(v) & \text{ if }v\in V(H)\\
		c'(v) & \text{ if }v\not\in V(H)
		\end{cases}$$
	\end{enumerate}
\end{defi}

Our constructions in the next sections depend on constructing partial colorings for rectangles with prescribed size, and then using them to define partial colorings of the square grid, for this the idea of \textbf{copy and paste} a coloring will be useful.

\begin{lemma}[Copy and Paste]\label{lemma:copy_paste}
	Let $R_1$ and $R_2$ be rectangles of dimensions $(m_1\times n_1)$ and $(m_2\times n_2)$ respectively. Suppose $c_1$ is a partial coloring of $R_1$ that realizes the color pairs $A\subset[k_0]^2$. If there exists $k\in \N$ such that $km_1 \leq m_2$ and $n_1+(k-1) \leq kn_2$, then we can define a partial coloring $c_2$ of $R_2$ that also realizes the color pairs $A$.
\end{lemma}

\begin{proof}
	If $k=1$ then $m_1\leq m_2$ and $n_1\leq n_2$, so embedding the rectangle $R_1$ in $R_2$ gives the desired partial coloring. Suppose now that $k\geq 2$. Let $R'_1$ be a rectangle of dimensions $(m_1\times n_1')$ where $n_1'=kn_2-(k-1)$. Embed $R_1$ into $R_1'$, this produces a partial coloring $c_1'$ of $R_1'$ which still realizes $A$. Consider now the subgraphs $H_1, H_2, \cdots, H_k$ of $R_1'$ induced by the sets of vertices $$V(H_r)=\{(i,j)| 1\leq i\leq m_1,\ (r-1)(n_2-1)+1\leq j\leq r(n_2-1)+1 \}$$
	
    \noindent for $1\leq r\leq k$. Each $H_r$ is a rectangle of dimensions $(m_1\times n_2)$. Note that all edges of $R_1'$ appear in $\cup_{r=1}^k E(H_r)$. 
	
	Now think of these as disjoint rectangles with the cells colors coming from the coloring $c_1'$ of $R_1'$. We can put $H_1$ to the side of $H_2$, to the side of $H_3$, and so on to the side of $H_k$ (see Figure \ref{fig:pasting}). This produces a rectangle of dimensions $(km_1\times n_2)$ that can be embedded into $R_2$, so the colored cells give a partial coloring $c_2$ of $R_2$. If a pair of colors in $A$ is realized by $c_1'$, meaning there is an edge with those colors in $R_1'$, then that edge also appears in one of the $H_r$'s, and thus it appears in $R_2$, so $c_2$ also realizes that pair of colors. Therefore, $c_2$ is a partial coloring of $R_2$ that realizes $A$. 
	
	While the proof is now completed, let us point out that when we embed $R_1$ into $R_1'$, the last $kn_2-n_1-(k-1)$ rows of $R_1'$ have no color assigned. When this number is less than $n_2$, this is a rectangular subgraph of $R_2$ of dimensions $(m_1\times(kn_2-n_1-(k-1)))$ that is empty. This rectangle occurs in the columns $(k-1)m_1+1,\cdots, km_1$ and the last $kn_2-n_1-(k-1)$ rows. Also, since the pasting only uses the first $km_1$ columns of $R_2$, the last $m_2-km_1$ columns remain empty as well. 
\end{proof}

\begin{figure}
	\centering
	\includegraphics[width=0.5\linewidth]{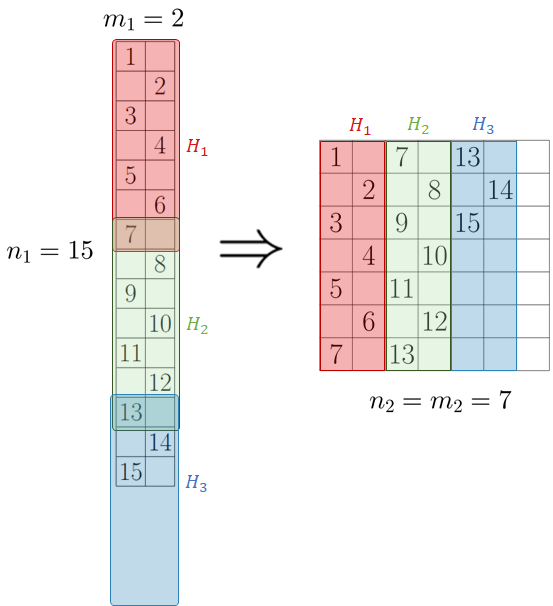}
	\caption{Copy and Paste process of a $(3\times 15)$ rectangle into a $(6\times 6)$ square. Here $k=3$.}
	\label{fig:pasting}
\end{figure}

\subsection{Achromatic Number of Paths and Roichmann Rectangles}

We start our exposition by computing the achromatic number of paths and what we refer as ``Roichman Rectangles''. While these are nice examples by themselves, we will use them later as some of our building blocks to produce complete colorings in the proof of our Theorem \ref{thm:improved_roichman}.

\begin{defi}
	For $n\geq 2$, the $n$-path is the graph $P_n$ with $(n+1)$ vertices $V(P_n)=\{0,1,\cdots,n\}$ and $n$ edges $u\sim v$ whenever $|u-v|=1$.
\end{defi}

\begin{theorem}\label{thm:achromatic_path}
	Given $n\leq 2$, let $q$ be the maximal integer such that $$\binom{q}{2}\leq n \hspace{1em}\text{and}\hspace{1em} q\left\lceil \frac{q-1}{2}\right\rceil\leq n+1,\text{ then}$$ $$\Psi(P_n)=q.$$
\end{theorem}

From Theorem \ref{thm:achromatic_path} and Lemma \ref{lemma:ineq1}, we get that $\Gamma(P_n)=\Psi(P_n)$ and both are optimal. 

Let us prove the following lemma, that will also be useful the next section. 

\begin{lemma}\label{lemma:extension_path}
	Let $k\geq2$. Then we can construct a coloring for the path $P_{M}$ that realizes all color pairs $[k]\times\{k+1\}$, where $M=3k/2$ if $k$ is even, and $M=(3k-1)/2$ if $k$ is odd. 
\end{lemma}
\begin{proof}
	Consider the coloring $f:V(P_M)\to [k+1]$ given by
	$$f(i)=\begin{cases}
	k+1, & \text{ if } 3|i\\
	2\left\lfloor\frac{i}{3}\right\rfloor + (i\mod 3), & \text{ if } 3\not| i
	\end{cases}$$
	All 3rd vertices starting at 0 get the color $k+1$, ending at the vertex $3k/2$ if $k$ is even and at vertex $3(k-1)/2$ if $k$ is odd. 
	Note in both cases the last vertex not divisible by 3 gets the color $k$, so each color in $[k]$ appears adjacent to the color $k+1$ proving the claim (see Figure \ref{fig:extension_path}).
\end{proof}

\begin{figure}
	\centering
	\begin{subfigure}[a]{0.75\textwidth}
		\includegraphics[width=1\textwidth]{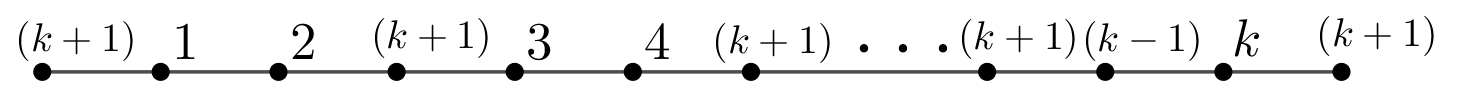}
		\caption{$k$ even}
	\end{subfigure}\\
	\begin{subfigure}[b]{0.75\textwidth}
		\includegraphics[width=1\textwidth]{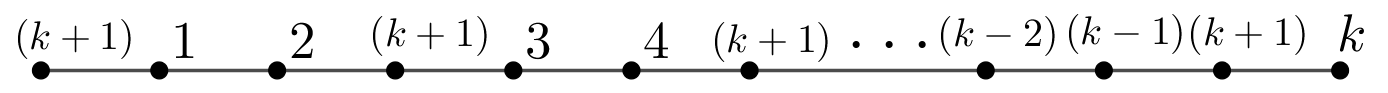}
		\caption{$k$ odd}
	\end{subfigure}
	\caption{Coloring of Paths realizing $[k]\times\{k+1\}$.}
	\label{fig:extension_path}
\end{figure}

To make our proof easier, we introduce the language of graph homomorphisms first. 

\begin{defi} Given two graphs $G$ and $H$, a \textbf{graph homomorphism} from $G$ to $H$ is a function $f:V(G)\to V(H)$ defined on the vertices of $G$ such that whenever $u\sim v$ in $G$ then also $f(u)\sim f(v)$ in $H$. 
\end{defi}

We can thus understand graph homomorphisms as functions that send vertices to vertices and edges to edges. Denoting by $K_k$ the $k$-complete graph, i.e. the graph with vertices $V(K_k)=[k]$ and edges all possible pairs of different vertices, we observe that a $k$-complete proper coloring $c:V(G)\to [k]$ can be seen as an edge-surjective graph homomorphism $c:G\to K_k$. 

\begin{proof}[Proof of Theorem \ref{thm:achromatic_path}]
	Note that any graph homomorphism $f:P_n\to H$ has as the image a $n$-walk from $f(0)$ to $f(n)$ in $H$. In particular, if $c:G\to K_k$ is a $k$-complete proper coloring, since it is edge-surjective, the corresponding $n$-walk on $K_k$ must traverse all edges. 
	
	Let $q$ be as defined in the statement of the theorem. Consider two cases:
	\begin{enumerate}
		\item Case $\mathbf{q}$ is \textbf{odd}. All vertices of $K_q$ have $q-1$ neighbors, so all have even degree. Recall that a graph has an \textit{Eulerian circuit} if and only if all vertices have even degree. Let $m=\binom{q}{2}$, and fix an Eulerian circuit $c_0, c_1, \cdots, c_m$ on $K_q$ that traverses each edge exactly once. Since by definition $m\leq n$, we can define the coloring $c:P_n\to K_q$ by
		$$c(i)=\begin{cases}
		c_i & \text{if } 0\leq i\leq m \\
		c(i-1)+1 \mod q & \text{if } m+1\leq i\leq n
		\end{cases} $$
		
		\item Case $\mathbf{q}$ is \textbf{even}. In this case $q-1$ is odd. Let $m=\binom{q-1}{2}$; just as above we can define a $(q-1)$-complete proper coloring of $P_{m}$ by using an Eulerian circuit on $K_{q-1}$, and choose it so that $c(0)=c(m)=q-1$. Using this coloring for the first $m$ vertices of the path $P_n$ produces a partial coloring realizing all pairs of colors in $[q-1]$, so we need to use the remaining vertices to realize the pairs $(j,q)$ for $1\leq j\leq q-1$. Let $H$ be the induced subgraph of $P_n$ on the set of vertices $\{m, m+1, \cdots, n\}$, so $H$ is a path with $(n+1)-m$ vertices. Note
		\begin{align*}
		    (n+1)-m &=(n+1)-\frac{(q-1)(q-2)}{2}\\
		    &\geq q\left\lceil\frac{q-1}{2}\right\rceil -\frac{(q-1)(q-2)}{2}\\
		    &=\frac{q^2}{2}-\frac{(q-1)(q-2)}{2}=\frac{3q-2}{2}.
		\end{align*}

		Using lemma \ref{lemma:extension_path}, we can define a coloring $c'$ on the path $P_M$, where $M=\frac{3q-4}{2}$, that realizes all pairs of colors we need. Define the \textit{reverse coloring} $\bar{c}$ as $\bar{c}(i)=c'(M-i)$, so $\bar{c}$ realizes the same pairs of colors and $\bar{c}(0)=q-1$. Since $P_M$ has at least $\frac{3q-2}{2}$ vertices, we can assign the first few vertices of the subgraph $H$ the corresponding colors of $\bar{c}$ in $P_M$, and note that this is possible since the color of vertex $m$ (the only vertex in $H$ that already had a color assigned) coincides with $\bar{c}$. Therefore, we have a partial coloring of $P_n$ realizing all pairs of colors in $[q]$. Finally, assign any proper coloring to the last few vertices of $P_n$ to produce the desired $q$-complete proper coloring. \qedhere
		\end{enumerate} 	
\end{proof}

The following rectangles and the explicit coloring that we present were first discovered by Roichman in \cite{Roichman1990}.

\begin{defi}
	Given $m\geq 4$, let $N=16(m-2)+2$. The corresponding \textbf{Roichman Rectangle} $R_m$ is the rectangular grid graph of dimensions $(N\times m)$.
\end{defi}

\begin{theorem}\label{thm:roichman1}
	Given $m\geq 4$. Let $\hat{\Psi}=8(m-2)$, then $$\hat{\Psi}\leq\Psi(R_m)\leq 8m-10.$$
\end{theorem}

\begin{proof}
	The upper bounds follows from Lemma \ref{lemma:ineq1} by noting $\binom{8m-10}{2}\leq |E(R_m)|<\binom{8m-9}{2}$. To prove the lower bound, Roichman introduced the following explicit partial coloring $c:V(R_m)\to\{0,\cdots,\hat{\Psi}\}$:
	
	$$c(i,j)=\begin{cases}
	\frac{i}{2}+3\frac{j}{2} \mod\hat{\Psi} & \text{ if } 2|(i+j)\\
	c(i-1,j)-4(j-1)+2\mod\hat{\Psi} & \text{ if } 2\not|(i+j), i\geq 2, \\
	&\text{ and} 2\leq j\leq m-1,\ 
	\end{cases}$$
	
	Note that $c$ leaves empty all odd vertices in rows 1 and $m$, and also in column 1. $c$ will be a complete coloring independently of the color of these, so to have a complete proper coloring, simply assign to them any color different to its neighbors. 
	
	If $(i,j)$ is an odd vertex of color $a$, with $2\leq j\leq m-1$, its neighbors have colors $a+4(j-1)-3$, $a+4(j-1)-2$, $a+4(j-1)-1$ and $a+4(j-1)$, all modulo $\hat{\Psi}$, this shows the coloring is always proper. Also note that in every fixed row $2\leq j\leq m-1$, the color of each parity of vertices increases by one, modulo $\hat{\Psi}$, as $i$ increases. Since in these rows there are $N/2=8m-15$ even vertices and $N/2-1=8m-16$ non-empty odd vertices, each parity of vertices attains all colors.
	
	We now prove that this is a complete partial coloring. Take any pair of different colors $a,b$, and suppose that $(b-a)\mod\hat{\Psi}\in\{1,\cdots,\hat{\Psi}/2\}$, otherwise since $a-b\equiv \hat{\Psi}-(b-a)\mod\hat{\Psi}$ simply switch $a$ and $b$. Let $x= (b-a)\mod\hat{\Psi}$, we can write $x+3=4q+r$, where $1\leq q\leq \hat{\Psi}/8= m-2$ and $0\leq r\leq 3$, hence $x=4[(q+1)-1]-(3-r)$, where $2\leq q+1\leq m-1$ and $0\leq 3-r\leq 3$. Then there is an odd vertex of color $a$ in row $q+1$, and one of its neighbors has color $a+x\equiv b\mod\hat{\Psi}$ as desired. 
\end{proof}

\subsection{Modified Roichman Rectangles and 2-Ribbons}
\begin{defi}
	Given $m\geq 3$, let $\bar{N}=16(m-1)+1$. The corresponding \textbf{Modified Roichman Rectangle} $\bar{R}_m$ is the rectangular grid graph of dimensions $(\bar{N}\times m)$.
\end{defi}

\begin{theorem}\label{thm:modified_roichman}
	Given $m\geq 3$. Let $\bar{\Psi}=8m-7$, then $$\bar{\Psi}\leq\Gamma(\bar{R}_m)\leq 8m-6.$$
\end{theorem}

\begin{proof}
	The upper bound follows from Lemma \ref{lemma:ineq1}, since $$|E(\bar{R}_m)|=32m^2-47m+15<32m^2-44m+21=\binom{8m-5}{2}.$$ 
	Let $\bar{N}=16(m-1)+1$, and consider the Roichman Rectangle $G=R_{m+1}$ of height $16((m+1)-2)+2=\bar{N}+1$ and width $m+1$. Note that dropping the first row and column of $G$, we get a copy of the Modified Roichman Rectangle $\bar{R}_m$. Formally, we consider the subgraph $H\subset G$ induced on the vertices $(i,j)$ with $2\leq i\leq \bar{N}+1$ and $2\leq j\leq m+1$, so $H=\bar{R}_m$. Let $c$ be the partial coloring in the proof of Theorem \ref{thm:roichman1}, \textit{partial} since it leaves the odd vertices in the first column and the outermost rows \textit{empty}. Consider now $\bar{c}$ the corresponding restriction of $c$ to $H$. Let $\hat{\Psi}=8(m-1)$, and recall $c$ is a $\hat{\Psi}$-complete coloring for $G$, however by restricting our attention to $H$, $\bar{c}$ is not complete anymore, let us compute its remainder. 
	
     Note that the color of an even vertex $(1,2r+1)$ in the first column, and the even vertex in the column $\bar{N}=16m-15$ at the same row, are the same:
	
	$$c(1,2r+1)=3r+2 \mod \hat{\Psi}$$ $$c(16m-15,2r+1)=8m-3r-6 \mod\hat{\Psi}$$
	
	Since $(8m-3r-6)-(3r+2)=8m-8=\hat{\Psi}$, $c(1,2r+1)=c(\bar{N},2r+1)$. And the odd vertices in the columns 2 and $\bar{N}+1$ are thus colored the same, since they only depend on their even neighbor to the left and the row number. Therefore, removing the first column removes no pair of adjacent colors, since these pairs also appear between the last two columns. 
	
	In the first row of $G$ only the even vertices are non-empty. Thus, if $(x,1)$ is an even vertex of color $a$, its only nonempty neighbor is the odd vertex $(x,2)$, which has color $(a-1)\mod\hat{\Psi}$. Hence, by removing this row, we remove the adjacent pairs of colors $A=\{\{a,a+1\mod\hat{\Psi}\}$ for $0\leq a\leq \hat{\Psi}-1\}$, so $\bar{c}$ is a $\hat{\Psi}$-complete partial coloring with remainder $A$.
	
	We now extend $\bar{c}$ to the empty odd vertices in the last row of $H$, Given an odd vertex $(i,m+1)$, $2\leq i \leq \bar{N}+1$, define
	
	$$\bar{c}(i,m+1)=\begin{cases}
	\hat{\Psi} +1 & \text{ if } c(i-1,m+1)\text{ is even}\\
	c(i-1,m+1) & \text{ otherwise.}
	\end{cases}$$ 
	Note that our construction doesn't give a proper coloring, because Theorem \ref{thm:modified_roichman} only includes the complete coloring number $\Gamma(\bar{R}_m)$ and not the achromatic number.  We now proceed to prove that $\bar{c}$ is a $(\hat{\Psi}+1)$-complete coloring. Recall that in the last row of $G$, the even vertices take on all the colors in $[\hat{\Psi}]$, and that they increment by 1, this is still the case for $H$, since the even vertices of the column 1 have the same colors as in the column $\bar{N}$. By the definition of $\bar{c}$, all even vertices with an even color have as the neighbor to the right the new color $\hat{\Psi}+1$; moreover, the next even vertex to the right must have the successor odd color, so this vertex has as the neighbor to the left the new color $\hat{\Psi}+1$. All odd colors appear like this, since addition modulo $\hat{\Psi}$ preserves the parity, and because the first even vertex of row $m+1$ has the same color as the last one. So far we know that all pairs of colors $\{x,\hat{\Psi}+1\}$ appear in $\bar{c}$. 
	Finally note that if color $2r+1$ appears in an even vertex, we copy the same color to the vertex to its right, and this one will have as the neighbor above color $2r$, and as the neighbor to the right $2r+2$. Once again, the same happens at the extremes of the row. So all pairs of colors $\{i,i+1\}$ for $i\in[\hat{\Psi}]$ are adjacent in $\bar{c}$ (see Figure \ref{fig:last_row_bar_c}). Therefore, $\bar{c}$ is a $(\hat{\Psi}+1)$-complete partial coloring, and so $\Gamma(\bar{R}_m)\geq \hat{\Psi}+1=8(m-1)+1=8m-7$.

\begin{figure}
	\begin{tabular}{m{1em}|m{3em}|m{3em}|m{3em}|m{3em}|m{3em}|m{3em}|m{1em}}
		\hline
		\cellcolor[HTML]{9B9B9B} & & \cellcolor[HTML]{9B9B9B}$2r-1$ &  & \cellcolor[HTML]{9B9B9B}$\phantom{xx}2r$ &                                & \cellcolor[HTML]{9B9B9B}$2r+1$& \\ \hline
		$\cdots$ &\multicolumn{1}{|c|}{\cellcolor[HTML]{9B9B9B}$2r$} & $\hat{\Psi}+1$                 & \cellcolor[HTML]{9B9B9B}$2r+1$ & $2r+1$                       & \cellcolor[HTML]{9B9B9B}$2r+2$ & $\hat{\Psi}+1$ & $\cdots$                \\ \hline
	\end{tabular}
\captionsetup{justification=centering}
\caption{Definition of coloring $\bar{c}$ at rows $m$ and $m+1$. The shaded cells are the even vertices, the color of the odd vertices in row $m$ are not displayed.}
\label{fig:last_row_bar_c}
\end{figure}
	
\end{proof}

We will use the following construction together with the Roichman Modified Rectangle to prove our theorem \ref{thm:improved_roichman}.

\begin{defi}
	Given $n\geq 1$, we call \textbf{2-Ribbon of length $n$} the graph grid of dimensions $(2\times n)$, which we will denote $L_n$. 
\end{defi} 

\begin{lemma}\label{lemma:2_ribbon}
	Given $k\geq 2$, we can define a coloring of the 2-Ribbon of lenght $k+3$, $L_{k+3}$, with the colors in $[k+3]$ that realizes all color pairs in $[k]\times\{k+1,k+2,k+3\}\cup\{(k+1,k+2),(k+1,k+3),(k+2,k+3)\}$.
\end{lemma}

\begin{proof}
	Consider the following coloring $c:V(L_{k+3})\to [k+3]$,
	$$c(i,j)=\begin{cases}
	j\mod k & \text{ if } 2| (i+j) \text{ and } 1\leq j\leq k+1\\
	k + (j\mod 3) & \text{ if } 2\not|(i+j)\text{ and } 1\leq j\leq k+2\\
	k + (k\mod 3) & \text{ if } 2|(i+j)\text{ and } j=k+2\\
	k + (k+1\mod 3) & \text{ if } 2|(i+j)\text{ and } j=k+3
	\end{cases}\text{ for } 1\leq j\leq k$$
	
	Where for simplicity we take $(3\mod 3)=3$ and $(k\mod k)=k$. Note that columns $k+1$, $k+2$ and $k+3$ are chosen to produce the pairs $(k+1,k+2)$, $(k+1,k+3)$ and $(k+2,k+3)$. All even vertices in rows $2\leq j\leq k+1$ have neighbors with colors $k+1, k+2$ and $k+3$, and these even vertices take once each color in $[k]$, producing then the remaining pairs of adjacent colors $[k]\times\{k+1,k+2,k+3\}$ (See Figure \ref{fig:2-ribbon}).
\end{proof}

\begin{figure}
	\begin{tabular}{|c|c|c|c|c|c|c|c|c|c|c|c|}
		\cline{1-7} \cline{9-12}
		\cellcolor[HTML]{9B9B9B}1 & k+2                       & \cellcolor[HTML]{9B9B9B}3 & k+1                       & \cellcolor[HTML]{9B9B9B}5 & k+3                       & \cellcolor[HTML]{9B9B9B}7 & $\cdots$                         & \cellcolor[HTML]{9B9B9B}k & k+3                       & \cellcolor[HTML]{9B9B9B}k+2 &                             \\ \cline{1-7} \cline{9-12} 
		k+1                       & \cellcolor[HTML]{9B9B9B}2 & k+3                       & \cellcolor[HTML]{9B9B9B}4 & k+2                       & \cellcolor[HTML]{9B9B9B}6 & k+1                       & \cellcolor[HTML]{FFFFFF}$\cdots$ & k+2                       & \cellcolor[HTML]{9B9B9B}1 & k+1                         & \cellcolor[HTML]{9B9B9B}k+3 \\ \cline{1-7} \cline{9-12} 
	\end{tabular}
	
	\captionsetup{justification=centering}
	\caption{Definition of coloring ${c}$ on the 2-ribbon $L_{k+2}$ (transposed to save space). The shaded cells are the even vertices.}
	\label{fig:2-ribbon}
\end{figure}

We now proceed to prove our main result Theorem \ref{thm:improved_roichman}.

\begin{proof}[Proof of Theorem \ref{thm:improved_roichman}]
	Since Lemma \ref{lemma:n_9_19} gives $\Gamma(G_n)\geq 2n-3$ for $n\leq 19$, we may suppose $n\geq 20$. Let $\Gamma_0(n)=8\left\lfloor\frac{n}{4}\right\rfloor-7$. First we will use the Modified Roichman Rectangles to prove that $\Gamma(G_n)\geq \Gamma_0(n)$. To do this, let $m=\left\lfloor\frac{n}{4}\right\rfloor$ and consider the Modified Roichman Rectangle $\bar{R}_m$ with the coloring $c$ we produced in Theorem \ref{thm:modified_roichman}. Thus $c$ in $\bar{R}_m$ realizes all pairs of different colors in $[8m-7]^2$, and has dimensions $(m\times (16m-15))$. Note that by letting $k=4$ we satisfy $4m\leq n$ and $(16m-15)+3\leq 4n$, so we can use the Copy and Paste lemma \ref{lemma:copy_paste} to produce a coloring of $G_n$ that realizes all pairs in $[8m-7]^2$. For simplicity let us call this $(8m-7)$-complete coloring also $c$. In what follows we will show how to extend this coloring to prove the theorem. 
	
	\begin{enumerate}
		\item $\mathbf{(n=4m)}$. Then $\Gamma_0=8\left(\frac{n}{4}\right)-7=2n-7$. We can see $c$ as a $(2n-6)$-complete partial coloring with remainder $\{2n+6\}\times[2n-7]$. Recall the comment in the proof of the Copy and Paste Lemma, there is a rectangle of dimensions $(m\times 12)$ \textit{empty} in $G_n$. From Lemma \ref{lemma:extension_path} by using $k=2n-6$, we can get a coloring of the path $P_M$ that realizes all color pairs $\{2n+6\}\times[2n-7]$, where $M=3n-9=12m-9$. So $P_M$ has $12m-8$ vertices, so we can embed it as a path into the empty rectangle of dimensions $(m\times 12)$, producing a $(2n-6)$-complete extension of $c$ (see Figure \ref{fig:embedding_paths}). The coloring $c$ and its extension produces a $(2n-6)$-complete partial coloring of $G_n$, so $\Gamma(G_n)\geq 2n-6$.		
		\item $\mathbf{(n=4m+1)}$. Then $\Gamma_0=8\left(\frac{n-1}{4}\right)-7=2n-9$. We can regard $c$ as a $(\Gamma_0+3)$-complete partial coloring with remainder $[2n-6]^2\setminus[2n-9]^2$. We will construct $c'$ a $(2n-6)$-complete extension of $c$. To do so, recall the comment in the proof of the Copy and Paste Lemma, there is a rectangle of dimensions $(m\times 16)$ \textit{empty} in $G_n$, and also the last column is empty. By combining the last 16 rows of the last column of $G_n$ together with said rectangle, we get an empty rectangle of dimensions $((m+1)\times 16)$. From Lemma \ref{lemma:2_ribbon}, we can get a 2-Ribbon $L$ of length $\Gamma_0+3=8m-4$ with a coloring $c_0$ that realizes precisely the pairs in $[2n-6]^2\setminus[2n-9]^2$. Let $M$ be a rectangle of dimensions $(16\times (m+1))$. Note that with $k=8$, $L$ and $M$ satisfy the inequalities of the Copy and Paste Lemma:
		$8\cdot 2\leq 16\text{ and } (8m-4)+7\leq 8(m+1).$
		So we can produce a coloring of $M$ that realizes $[2n-6]^2\setminus[2n-9]^2$. Finally, transposing $M$ we can embed it into the empty rectangle in $G_n$, thus producing the $(2n-6)$-complete extension of $c$ that we wanted, hence $\Gamma(G_n)\geq 2n-6$.
		
		\item $\mathbf{(n=4m+2}$ \textbf{or} $\mathbf{4m+3)}$. We will combine the ideas of the previous two cases to extend the partial coloring $c$ into a $(8m-3)$-complete partial coloring. As before, the empty rectangle in $G_n$ has at least 20 rows (here we use $n\geq 20$) and $m$ columns. There are also at least 2 empty columns at the side of this rectangle, of height $n\geq 4m+2$. This rectangle and this 2 columns form a connected subgraph $L$ of $G_n$ (see Figure \ref{fig:embedding_pathsb}). We now will produce a partial coloring on $L$ that realizes all the pairs $[8m-3]^2\setminus[8m-7]^2$. From Lemma \ref{lemma:extension_path}, we can color a path $P_M$, where $M=\frac{3(8m-7)-1}{2}=12m-11$, such that it realizes all pairs $[8m-6]^2\setminus[8m-7]^2$. This path has $12m-10$ vertices, an can be embedded in the graph $L$ as the path in Figure \ref{fig:embedding_pathsb}, occupying all vertices of $L$ except for the subgraph $Q$, where $Q$ is a $((m+1)\times 16)$ rectangle without the cells at positions $(m,15)$ and $(m,16)$. Let $Q^*$ be a rectangle of dimensions $(16\times (m+1))$. By Lemma \ref{lemma:2_ribbon} we can get a 2-Ribbon of length $8m-3$ that realizes all pairs in $[8m-3]^2\setminus[8m-6]^2$, setting $k=8$, we satisfy the inequalities from the Copy and Paste Lemma, so we have a coloring of $Q^*$ that realizes said pairs. Note that, from the comment of the Copy and Paste Lemma, we get an empty rectangle at the corner of $Q^*$ of dimensions $(2\times4)$. Removing those vertices from $Q^*$ and transposing it, we can fit it into the subgraph $Q$. This last embedding and the embedding of $P_M$ together, produce a $(8m-3)$-complete extension for the partial coloring $c$, hence proving
		$$\Gamma(G_n)\geq 8m-3=\begin{cases}
		2n-7 & \text{ if } n\equiv 2\mod 4\\
		2n-9 & \text{ if } n\equiv 3\mod 4 \qedhere
		\end{cases} $$
	\end{enumerate}	
\end{proof}

\begin{figure}
	\centering
	\begin{subfigure}[t]{0.45\textwidth}
		\centering
		\includegraphics[width=0.7\textwidth]{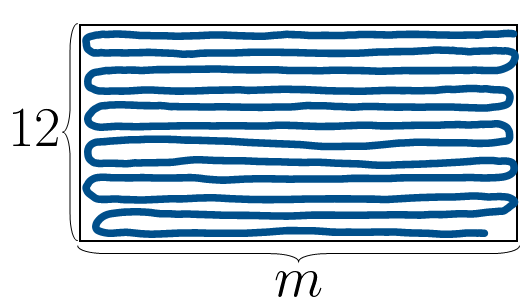}
		\caption{Embedding $P_M$ in a rectangle $(m\times 12)$.}
	\end{subfigure}%
	\begin{subfigure}[t]{0.45\textwidth}
		\centering
		\includegraphics[width=1\textwidth]{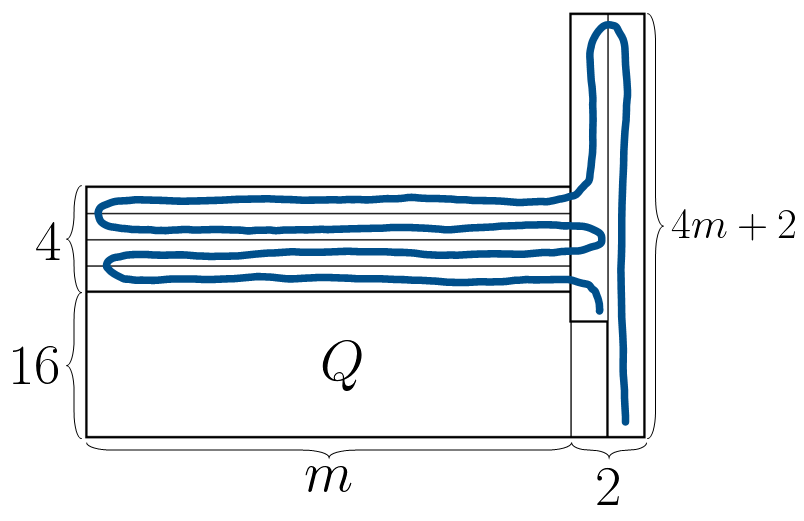}
		\caption{Subgraphs $L$  and $Q$ with embedding of $P_M$.}\label{fig:embedding_pathsb}
	\end{subfigure}
	\caption{Embedding paths for Theorem \ref{thm:improved_roichman}}
	\label{fig:embedding_paths}
\end{figure}

\section{Higher Dimensions}
A natural follow-up question is to find the complete coloring number for $d$-dimensional square grids or lattices, for $d\geq3$. Since a $d$-dimensional grid of dimensions $n\times\dots\times n$ has $dn^{d-1}(n-1)$ edges, from Lemma \ref{lemma:ineq1}, we see that the number $$q^d(n)=\max\left\{k: \binom{k}{2}\leq dn^{d-1}(n-1)\right\}$$ is an upper bound for the complete coloring number. If $G_n^d$ denotes the $d$-dimensional grid, the question becomes: Is $\Gamma(G^d_n)=q^d(n)$, for every $n\geq n_0$?

Edwards work in \cite{Edwards2009} shows this is again true for $n_0$ big enough, but again leaves the question open for small values of $n$. When $d\geq 3$ the problem quickly becomes intractable even for small values of $n$, Table \ref{fig:3D} summarizes the best complete colorings we were able to find for $n\leq 8$.

\begin{table}[]
\centering
\begin{tabular}{|c|c|c|}
\hline
$n$ & $\Gamma(G^3_n)\geq$ & $q(n)$ \\ \hline
3   & 10                  & 10     \\
4   & 16                  & 17     \\
5   & 23                  & 25     \\
6   & 30                  & 33     \\
7   & 39                  & 42     \\
8   & 48                  & 52     \\ \hline
\end{tabular}
\caption{Bounds found for 3D Grids.}\label{fig:3D}
\end{table}

When $n=2$, all vertices have degree 3, and the upper bound is obtained from item 3. of Lemma \ref{lemma:ineq1}: $\max\left\{k: k\left\lceil\frac{k-1}{3}\right\rceil\leq 8\right\}=4$, so $\Gamma(G_2^3)\leq 4$ which can be easily attained. Hence the only two complete coloring numbers we know for the 3-dimensional case are $\Gamma(G_2^3)=4$ and $\Gamma(G_3^3)=10$.

\nocite{*}
\bibliographystyle{alpha}
\bibliography{References}

\end{document}